\documentclass[12pt]{article}
\usepackage{amsthm, amsmath}
\usepackage{tikz}
\usepackage{fullpage}

\newtheorem{thm}{Theorem}

\newtheorem{con}[thm]{Conjecture}

\newtheorem{ques}{Question}

\begin{document}
\title{Two conjectured strengthenings of Tur\'an's theorem}
\author{
    Clive Elphick\thanks{\texttt{clive.elphick@gmail.com}, 
    School of Mathematics, University of Birmingham, Birmingham, UK.}
    \quad 
    William Linz\thanks{\texttt{wlinz@mailbox.sc.edu}, 
    Department of Mathematics, University of South Carolina, Columbia, South Carolina, 29208, USA. This work was performed while the author was at the University of Illinois and was partially supported by NSF RTG Grant DMS-1937241}
    \quad
    Pawel Wocjan\thanks{\texttt{Pawel.Wocjan@ibm.com}, 
    IBM Quantum, IBM T.J. Watson Research Center, Yorktown Heights, NY 10598, USA.} 
}

\maketitle

\abstract{We investigate two conjectured spectral graph theoretic strengthenings of Tur\'an's theorem. Let $\mu_1 \ge \ldots \ge \mu_n$ denote the eigenvalues of a graph $G$ with $n$ vertices, $m$ edges and clique number $\omega(G)$.

The concise version of Tur\'an's theorem is that $n/(n - d)$ is a lower bound for the clique number $\omega(G)$, where $d$ is the average degree. Our first conjecture is that $d$ can be replaced in this bound with $\sqrt{s^+}$, where $s^+$ is the sum of the squares of the positive eigenvalues. We prove this conjecture for triangle-free, weakly perfect and Kneser graphs and for almost all graphs. We have also used various software tools to search for a  counter-example.

Nikiforov proved a spectral version of Tur\'an's theorem that
\[
\mu_1^2 \le \frac{2m(\omega(G) - 1)}{\omega(G)}, 
\]
and Bollob\'as and Nikiforov conjectured that for  $G \not = K_n$
\[
\mu_1^2 + \mu_2^2 \le \frac{2m(\omega(G) - 1)}{\omega(G)}.
\]
For our second conjecture, we propose that for all graphs $(\mu_1^2 + \mu_2^2)$ in this inequality can be replaced by the sum of the squares of the $\omega(G)$ largest eigenvalues, provided they are positive. We prove the conjecture for weakly perfect, Kneser, and classes of strongly regular graphs. We also provide experimental evidence and describe how the bound can be applied. Liu and Ning \cite{liu23} published a wide-ranging paper entitled ``Unsolved Problems in spectral graph theory'', and these two conjectures were placed second and fourth in their list of such problems.
}

\section{Notation}

Let $G$ be  a graph with no isolated vertices, with $n$ vertices, $m$ edges, average degree $d$, chromatic number $\chi(G)$ and clique number $\omega(G)$. We also let $A$ denote the adjacency matrix of $G$ and let $\mu = \mu_1 \ge \mu_2 \ge \ldots \ge \mu_n$ denote the eigenvalues of $A$. The inertia of $A$ is the ordered triple $(n^+, n^-, n^0)$, where $n^+$, $n^-$ and $n^0$ are the numbers counting multiplicities of positive, negative and zero eigenvalues of $A$, respectively. Let
\[
s^+ = \sum_{i=1}^{n^+} \mu_i^2 \quad \mbox{and} \quad
s^- = \sum_{i=n- n^- +1}^n \mu_i^2. 
\]

Note that:
\[
\sum_{i=1}^n \mu_i^2 = \mathrm{tr}(A^2) = 2m = s^+ + s^-. \nonumber
\]

\section{Replacing $\mu^2$ with $s^+$}

The maximum eigenvalue $\mu$ appears in many spectral bounds on graph parameters.

For instance, Edwards and Elphick \cite{edwards83} proved that
\[
\frac{2m}{2m - \mu^2} \le \chi(G),
\]
and Ando and Lin \cite{ando15} proved a conjecture due to Wocjan and Elphick \cite{wocjan13} that
\[
\frac{2m}{2m - s^+} = 1 + \frac{s^+}{s^-} \le \chi(G).
\]
Coutinho and Spier \cite{coutinho} recently proved the stronger result that $1 + s^+/s^- \le \chi_v(G)$, the vector chromatic number.

As another example of replacing $\mu^2$ with $s^+$, Hong \cite{hong93} proved for graphs with no isolated vertices that $\mu^2 \le 2m - n + 1$, and Elphick \emph{et al.} \cite{elphick16} conjectured that for all connected graphs $s^+ \le 2m - n + 1$. Similarly, Favaron \emph{et al.} \cite{favaron93} proved that $\omega(G) \le 2m/\mu$, and Wu and Elphick \cite{wu17} strengthened this result by proving that $\chi(G) \le 2m/\sqrt{s^+}$ (note that both $\omega(G) \le \chi(G)$ and $1/\sqrt{s^+} \le 1/\mu$). As a fourth example of replacing $\mu^2$ with $s^{+}$, Stanley \cite{stanley87} proved that 

\[
\mu \le \frac{\sqrt{8m + 1} - 1}{2},
\]
and Wu and Elphick \cite{wu17} used the result of Ando and Lin \cite{ando15} to prove the stronger bound that
\[
\sqrt{s^+} \le \frac{\sqrt{8m + 1} - 1}{2}.
\]

So in all of these cases we can (at least conjecturally) strengthen known spectral bounds by replacing $\mu^2$ with $s^+$. The next section considers the same replacement for a well known lower bound for the clique number.

\section{Conjectured strengthening of Wilf's  bound}

The concise version of Tur\'an's theorem is that $n/(n - d) \le \omega(G)$. Wilf \cite{wilf86} strengthened this bound by proving that:

\begin{equation}\label{wilf}
\frac{n}{n - \mu} \le \omega(G).
\end{equation}

This bound was further strengthened by Nikiforov \cite{nikiforov02} who proved the following conjecture of Edwards and Elphick \cite{edwards83}.

\begin{equation}\label{nikiforov}
\frac{2m}{2m - \mu^2} \le \omega(G).
\end{equation}

Wocjan and Elphick \cite{wocjan13} noted that
\[
\frac{2m}{2m - s^+} \not\le \omega(G).
\]

An alternative strengthening of bound (1) is provided by the following conjecture, which we have tested against the thousands of named graphs with up to 100 vertices in the Wolfram Mathematica database, and found no counter-example. Aouchiche \cite{aouchiche16} has tested this conjecture using his powerful AGX software, and also found no counter-example. Liu and Ning \cite{liu23} used SageMath to confirm the conjecture for all graphs with up to 10 vertices, and placed this conjecture second in their extensive list of unsolved problems in spectral graph theory. Note that for $d$-regular graphs, the left-hand side in bounds (\ref{wilf}) and (\ref{nikiforov}) is equal to $n/(n - d)$ whereas the left-hand side in Conjecture~\ref{first} exceeds $n/(n - d)$ for almost all regular graphs.

\begin{con}\label{first}
For any graph $G$
\[
\frac{n}{n - \sqrt{s^+}} \le \omega(G).
\]
This conjecture is exact, for example,  for complete regular multipartite graphs.
\end{con}

\begin{ques}
If Conjecture~\ref{first} is true, are there graphs, other than complete regular multipartite graphs, for which the conjecture is exact?
\end{ques}

 Note that it is immediate that for any graph $G$: 
\[
\frac{n}{n - \sqrt{s^-}} \le \omega(G)
\]
since Elphick \emph{et al} \cite{elphick16} proved that for all graphs $s^- \le n^2/4$.

We can prove Conjecture~\ref{first} for the following classes of graphs.

\subsection{Triangle-free graphs}
\begin{proof}

Let $t$ denote the number of triangles in a graph. It is well known that:
\[
\sum_{i=1}^n \mu_i^3 = \mathrm{tr}(A^3) = 6t,  
\]
so for triangle-free graphs
\[
\sum_{i=1}^{n^+} \mu_i^3 = -\sum_{i=n- n^- +1}^n \mu_i^3.
\]

Therefore, using that $\mu \ge |\mu_n|$
\[
s^- \ge \frac{\sum_{i=n- n^- +1}^n \mu_i^3}{\mu_n} = \frac{\sum_{i=1}^{n^+}\mu_i^3}{|\mu_n|} \ge \frac{\mu^3}{|\mu_n|} \ge \mu^2.
\]

Therefore, using that $\mu \ge 2m/n$

\[
\frac{4m^2}{n^2}s^+ \le \mu^2s^+ \le s^-s^+ = (2m - s^+)s^+ \le m^2 \mbox{  because  } s^+ + s^- = 2m.
\]

Hence for triangle-free graphs, $\sqrt{s^+} \le n/2$ which completes the proof.
\end{proof}
\subsection{Weakly perfect graphs}
\begin{proof}
Weakly perfect graphs have $\omega(G) = \chi(G)$. Therefore using the result due to Ando and Lin \cite{ando15} discussed above and that $\mu \ge 2m/n$:

\[
\frac{n}{n - \sqrt{s^+}} \le \frac{2m}{2m - s^+} \le \chi(G) = \omega(G).
\]

\end{proof}

\subsection{Kneser graphs} 

The Kneser graph $KG_{p,k}$ is the graph whose vertices correspond to the $k$-element subset of a set of $p$ elements, in which two vertices are joined if and only if the corresponding sets are disjoint. It is well known that:

\[
n = \binom{p}{k},\,  \omega(KG_{p, k}) = \left\lfloor\frac{p}{k}\right\rfloor,\, 2m = \binom{p}{k} \binom{p - k}{k} \mbox{ and } p \ge 2k.
\]
The eigenvalues (see Godsil and Royle \cite{godsil}) are :

\[
(-1)^i \binom{p - k - i}{k - i} \mbox{ with multiplicity } \binom{p}{i} - \binom{p}{i - 1}, \mbox{ for } i = 0,1,2,\ldots,k.
\]

\subsubsection{Proof for the case $k \le 2$}

\begin{proof}
The Kneser graphs with $k = 1$ are complete graphs. The Kneser graphs with $k = 2$ are strongly regular, with only three distinct eigenvalues. For these graphs

\[
n = \binom{p}{2},\, \omega(KG_{p, 2}) =  \left\lfloor\frac{p}{2}\right\rfloor,\, 2m = \binom{p}{2} \binom{p - 2}{2} \mbox{ and } p \ge 4.
\]

The eigenvalues with $k =  2$ are :

\[
(-1)^i \binom{p - 2 - i}{2 - i} \mbox{ with multiplicity } \binom{p}{i} - \binom{p}{i - 1}, \mbox{ for } i = 0,1,2. 
\]

Note that $\binom{p}{-1}$ is defined to equal zero. We are seeking to prove that:

\[
\frac{n}{n - \sqrt{s^+}} \le \frac{p - 1}{2} \le \left\lfloor\frac{p}{2}\right\rfloor = \omega(KG_{p,2}),
\]

which rearranges to

\[
s^+ = 2m - s^- \le \frac{n^2(p - 3)^2}{(p - 1)^2} = \frac{p^2(p - 3)^2}{4}.
\]

Inserting the negative eigenvalues this becomes:

\[
\binom{p}{2} \binom{p - 2}{2} - (p - 1)\binom{p - 3}{1}^2 \le \frac{p^2(p - 3)^2}{4}.
\]

Simple algebra reduces this  to 

\[
2p^2 - 9p + 6 \ge 0
\]

which is true for all $p \ge 4$. 

\end{proof}

\subsubsection{Proof for the case $k \ge 3$}

\begin{proof}

For $k \ge 3$ we can adopt a less detailed approach. We are seeking to prove that:

\[
\frac{n}{n - \sqrt{s^+}} \le \frac{n}{n - \sqrt{2m}} \le \frac{p - k + 1}{k} \le \left\lfloor\frac{p}{k}\right\rfloor = \omega(KG_{p,k})
\]

Let $p = 2k + s$, where $s \ge 0$. The required inequality then simplifies to:

\[
\sqrt{2m}(k + s + 1) \le n(s + 1).
\]
Then squaring both sides and substituting values for $2m$ and $n$ gives that:

\[
(k + s + 1)^2((k + s)!)^2 \le (2k + s)!s!(s + 1)^2. 
\]

With $s = 0$ this is true for all $k \ge 3$. With $s \ge 1$ and $k \ge 3$ the inequality becomes stronger.
\end{proof}

\subsection{Proof for almost all graphs}

We use the Erd\H{o}s-Renyi random graph $G_p(n)$, which consists of all graphs with $n$ vertices in which edges are chosen independently with probability $p$.  Bollob\'as and Erd\H{o}s \cite{bollobas76} proved that the clique number is almost always $x$ or $x+1$ where

\[
x = \frac{2\log n}{\log (1/p)} + O(\log \log n).
\]

Since almost all graphs have all degrees very close to $n/2$ we let $p = 0.5$. Therefore

\[
s^+ \le 2m \approx n^2/2. 
\]

So for almost all graphs:

\[
\frac{n}{n - \sqrt{s^+}} \le \frac{n}{n - n/\sqrt{2}} \approx 3.4 <\frac{2\log n}{\log 2} \approx \omega(G).
\]

\section{Generalising a conjecture due to Bollob\'as and Nikiforov}

We propose the following conjecture that provides a lower bound for $\omega(G)$ using up to  $\omega(G)$ eigenvalues. It considers the sum of the squares of the $\omega(G)$ largest eigenvalues, provided they are positive. Conjectures 2 and 3 were placed fourth and third respectively in \cite{liu23}.

\begin{con}\label{con:ew}
For any non-empty graph $G$, the clique number $\omega(G)$ satisfies the bound
\begin{equation}\label{elphick}
\mu_1^2 + \mu_2^2 + \ldots + \mu_\ell^2 \le \frac{2m(\omega(G) - 1)}{\omega(G)}, 
\quad\mbox{where } \ell = \min{(n^+ , \omega(G))}.
\end{equation}

This conjecture is exact, for example, for complete regular multipartite graphs.
\end{con}

\noindent
The obvious question is: why do we set $\ell = \min{(n^+ , \omega(G))}$? The answer, referring to Section~\ref{subsec:wpg} below, is that if $\ell = n^+$ then we have a lower bound for $2m(\chi - 1)/\chi$, so $\ell \le n^+$. Below we provide an example showing that $\ell$ must be chosen such that $\ell \le \omega(G)$. 

Conjecture~\ref{con:ew} strengthens the following conjecture due to Bollob\'as and Nikiforov~\cite{bollobas07}:

\begin{con}[Bollob\'as and Nikiforov]\label{con:bn}
Let $G$ be a $K_{\omega(G)+1}$-free graph of order at least $\omega(G) + 1$ with $m$ edges. Then
\begin{equation}\label{bollobas}
\mu_1^2 + \mu_2^2 \le \frac{2m(\omega(G) - 1)}{\omega(G)}.
\end{equation}
\end{con}

\noindent
As noted in \cite{bollobas07}, Conjecture~\ref{con:bn} is not valid for complete graphs. The conjecture also places a lot of emphasis on $\mu_2$. Zhang \cite{zhang} proved Conjecture 3 for regular graphs, using that $\mu_2 > 0$ in his proof. Lin, Ning and Wu \cite{lin20} proved the conjecture for triangle-free graphs. It should be noted that the cycle $C_7$ provides an example of a triangle-free graph for which
\[
\mu_1^2 + \mu_2^2 +\mu_3^2 > \frac{2m(\omega(G) - 1)}{\omega(G)}.    
\]
This example demonstrates that $\ell$ must be chosen such that $\ell\le\omega(G)$. It should be noted that there exist graphs $G$ with $n^+ < \omega(G) < \chi(G)$. For example, the complement of the folded 7-cube on 64 vertices has $n^+ = 8, \omega(G) = 22$ and $\chi(G) = 32$.

The cycle $C_5$ provides an example of  a triangle-free graph for which $\mu_n^2 + \mu_{n-1}^2 > m.$ A proof of Conjecture 2 or 3 will therefore need to incorporate asymmetry between positive and negative eigenvalues. For example, the proof in \cite{lin20} uses that $\mu_1 \ge -\mu_n$. This contrasts with the symmetry between positive and negative eigenvalues to be found in some lower bounds for the chromatic number (\cite{ando15},\cite{elphick17},\cite{wocjan13}).

We have the following experimental and theoretical results to support the conjecture. As proved in Section~\ref{subsec:wpg}, Conjecture~\ref{con:ew} is true when $\omega(G) = \chi(G)$. In searching for a counterexample we therefore focus on graphs with $\omega(G) < \chi(G)$. We also seek proofs for graph families with $\omega(G) < \chi(G)$, but it is hard to find such families where both $\omega(G)$ and the spectrum are known.

\subsection{Experimental evidence for Conjecture 2}

 There are over 3000 graphs with $\omega(G) < \chi(G)$ in the Wolfram Mathematica database with up to 50 vertices, and we found no counterexample. For about 35\% of these graphs, Conjecture~\ref{con:ew} outperforms Conjecture~\ref{con:bn}. 

 For all graphs (see Section~\ref{subsec:wpg}):

 \[
s^+ \le \frac{2m(\chi - 1)}{\chi}, \mbox{  but for some graphs  } s^+ > \frac{2m(\omega(G) - 1)}{\omega(G)}.
\]

Of the graphs we tested with $\omega(G) < \chi(G)$, there are 167 with $s^+ > \frac{2m(\omega(G) - 1)}{\omega(G)}$. These include $C_7$ and the Coxeter graph.

Additionally, using Sage we tested Conjecture~\ref{con:ew} for a certain subset of gcd graphs (see \cite{basic11} for a discussion of the clique number of these graphs). Specifically, we considered the graphs $X_n(d_1, d_2)$, where $n > 1$, $d_1$ and $d_2$ are proper divisors of $n$, and $V(X_n(d_1, d_2)) = \{0, \ldots, n-1\}$, $E(X_n(d_1, d_2)) = \{\{a, b\}| \gcd(a-b, n) \in \{d_1, d_2\}\}$. For $n \le 146$, we found no counterexample to Conjecture~\ref{con:ew}. It should be noted that many, but not all, of these gcd graphs $G$ do have $\chi(G) = \omega(G)$.  

\subsection{Proofs for various families of graphs}

\subsubsection{Weakly perfect graphs and triangle-free graphs}\label{subsec:wpg}

A weakly perfect graph has $\omega(G) = \chi(G)$. We can prove Conjecture~\ref{con:ew} for these graphs as follows.

\begin{proof}

 As discussed in Section 2:
\[
s^+ \le \frac{2m(\chi - 1)}{\chi}.
\]
Consequently for any weakly perfect graph where $\ell = \min{(n^+, \omega(G))}:$
\[
\mu_1^2 + \mu_2^2 + \ldots + \mu_\ell^2 \le s^+ \le \frac{2m(\chi - 1)}{\chi} = \frac{2m(\omega(G) - 1)}{\omega(G)}.
\]

\end{proof}

The family of weakly perfect graphs includes all bipartite graphs. More generally, the result of Lin, Ning and Wu~\cite{lin20} shows that Conjecture~\ref{con:ew} holds for all triangle-free graphs. 

\subsubsection{Kneser graphs}

In addition to the parameters of Kneser graphs discussed in Section 3.3.1, we also use that $\chi = p - 2k + 2$  and the following result for $n^+$  (see Section 2.10 of \cite{godsil15}).

\[
 n^+ = {p - 1 \choose k} \mbox{ for even } k \mbox{ or } {p - 1 \choose k - 1} \mbox{ for odd } k,\\
\]

 \begin{proof}

We are seeking to prove for $KG_{p,k}$ that:

\begin{equation}\label{eqn:eigsineq}
\mu_1^2 + \mu_2^2 + ... + \mu_{\ell}^2 \le \frac{2m(\omega(G) - 1)}{\omega(G)}, \mbox{  where  } \ell = \min{(n^+ , \omega(G))}.
\end{equation}
 
Note that $\omega(G) \le n^{+}(G)$. The left hand side of \eqref{eqn:eigsineq} becomes
\[\binom{p-k}{k}^2 + \left(\left\lfloor\frac{p}{k}\right\rfloor-1\right)\binom{p-k-2}{k-2}^2 \le \binom{p-k}{k}^2+\left(\frac{p}{k}-1\right)\binom{p-k-2}{k-2}^2.\]
Using $\binom{p-k-2}{k-2} = \frac{(k-1)k}{(p-k-1)(p-k)}\binom{p-k}{k}$, the previous expression can be rewritten as
\[\binom{p-k}{k}^2\left(1+\frac{k(k-1)^2}{(p-k-1)^2(p-k)}\right).\]

For the right hand side of \eqref{eqn:eigsineq}, we have
\[2m\left(1-\frac{1}{\omega(G)}\right) = \binom{p}{k}\binom{p-k}{k}\left(1-\frac{1}{\lfloor{\frac{p}{k}\rfloor}}\right) \ge \binom{p}{k}\binom{p-k}{k}\left(\frac{p-2k+1}{p-k+1}\right).\]

Hence, to prove \eqref{eqn:eigsineq}, it suffices to prove
\begin{equation}\label{eqn:newineq}
\binom{p-k}{k}^2\left(1+\frac{k(k-1)^2}{(p-k-1)^2(p-k)}\right) \le \binom{p}{k}\binom{p-k}{k}\left(\frac{p-2k+1}{p-k+1}\right).
\end{equation}

Cancelling and rearranging terms, \eqref{eqn:newineq} is equivalent to 
\[\binom{p-k+1}{k}\left(1+\frac{k(k-1)^2}{(p-k-1)^2(p-k)}\right) \le \binom{p}{k}.\]

To avoid trivialities, we henceforth assume $k\ge 3$; for $k=2$, the previous inequality can be easily shown directly. The previous inequality is in turn equivalent to 
\begin{equation}\label{eqn:newineq2}
(p-k)\cdots(p-2k+2)\left(1+\frac{k(k-1)^2}{(p-k-1)^2(p-k)}\right) \le p\cdots(p-k+2).
\end{equation}

Upon expansion, the left hand side of \eqref{eqn:newineq2} is 

\begin{align*}
&(p-k)\cdots(p-2k+2)+(p-k-2)\cdots(p-2k+2)\frac{k(k-1)^2}{p-k-1}\\
&\le (p-k)\cdots(p-2k+2)+k(k-1)(p-k-2)\cdots(p-2k+2)\\
&=(p-k-2)\cdots(p-2k+2)((p-k)(p-k-1) + k(k-1))\\
&=(p-k-2)\cdots(p-2k+2)(p^2-(2k+1)p+2k^2),\\
\end{align*}
where in the second line we use $p-k-1 \ge k-1$, which is valid for $p\ge 2k$. 

But now \eqref{eqn:newineq2} is easy to establish, since $p(p-1) \ge p^2-(2k+1)p+2k^2$ is valid even for $p\ge k$, and for any $i$, $p-i \ge p-k-i$.
\end{proof}

 \subsubsection{Special strongly regular graphs (SRGs) }

 \medskip
 \noindent

 \medskip
 \noindent
 We do not know how to prove Conjecture~\ref{con:ew} for all SRGs, but below are proofs for two classes of SRGs. Let $G = SRG(n, d, \lambda, \mu)$ with restricted eigenvalues $r > s$. 
 
 Cvetkovic \cite{cvetkovic71} proved that $\omega(G) \le \min{(n_{\mu_i\le -1} + 1 , n_{\mu_i\ge -1})}$, where $n_{\mu_i \le -1}$ and $n_{\mu_i \ge -1}$ denote the number of eigenvalues $\le -1$ and $\ge -1$ respectively. Therefore for SRGs:
 
 \[
 \omega(G) \le \min{(n^- + 1 , n^+)} \le n^+; \mbox{  so  } \ell = \min{(n^+, \omega(G))} = \omega(G).
 \]
 
 Therefore Conjecture~\ref{con:ew} simplifies as follows for SRGs:
 
 \[
 d^2 + (\omega(G) - 1)r^2 \le \frac{nd(\omega(G) - 1)}{\omega(G)},
 \]
 which provides a quadratic inequality for $\omega(G)$. 
 
 For example the Schlafli graph $SRG(27, 16, 10, 8)$ has spectrum $(16^1, 4^6, -2^{20})$ and $\omega(G) = 6$, so Conjecture~\ref{con:ew} becomes:
 
 \[
 d^2 + (\omega(G) - 1)r^2 = 336 \le \frac{nd(\omega(G) - 1)}{\omega(G)} = 360.
 \]
 
 The difficulty of proving Conjecture~\ref{con:ew} or ~\ref{con:bn} for SRGs is that we do not have a formula for the clique number, although it is known that for all SRGs, $\omega(G) \le 1 - d/s \le \chi(G)$. 
 
 Roberson \cite{roberson19} demonstrated experimentally that numerous SRGs have $\omega(G) = 1 - d/s < \chi(G)$, which he terms Type C SRGs. We can prove Conjecture~\ref{con:ew} for these SRGs as follows.

 \begin{proof}
 We wish to show that
 
 \begin{equation}\label{srg}
 d^2 + (\omega(G) - 1)r^2 = d^2 - \frac{dr^2}{s} \le \frac{nd(\omega(G) - 1)}{\omega(G)} = \frac{nd^2}{d - s}.
 \end{equation}
 
 Haemers \cite{haemers79} and others have noted that $(d - r)(d - s) = n(d + rs)$, and $\mu = rs + d$, so (\ref{srg}) simplifies to:
 
 \begin{equation}\label{typecsrgineq}
 r\mu \le ds(s + 1).
 \end{equation}
 
 First, if $\lambda \le \mu$, then $|r| \le |s|$, so in this case the inequality (\ref{typecsrgineq}) easily follows from $\mu\le d$ and $r\le s(s+1)$. Hence, we may assume that $\lambda > \mu$, which in particular implies that $f < g$, where $f$ and $g$ are the multiplicities of the eigenvalues $r$ and $s$, respectively. 

Now, using the bound $\omega(G) \le \min(n^{-} + 1, n^{+}) = 1 + \min(f, g) = 1+f$, combined with the assumption that $G$ is a Type C SRG, and so $\omega(G) = 1 -\frac{d}{s}$, we obtain that $G$ must satisfy
\begin{equation}\label{dsfineq}
d+sf \le 0.
\end{equation}

We will now assume for contradiction that (\ref{typecsrgineq}) does not hold for $G$, and show that this assumption leads to $d+sf > 0$, in direct contradiction to (\ref{dsfineq}). Indeed, using the identities $\mu = rs + d$ and $d+rf+sg = 0$, the negated form of (\ref{typecsrgineq}) becomes
\[r(rs+d) > ds(s+1) = (-rf-sg)s(s+1).\]
Expanding both sides of the preceding inequality and rearranging gives
\[r(d+sf) > -r^2s-rfs^2-s^3g-s^2g.\]
We focus now on the right-hand side of this inequality. Since by assumption $g > f$, and clearly $g \ge 1$, we find that
\[-r^2s-rfs^2-s^3g-s^2g > -g(r^2s+rs^2+s^3+s^2) = -g(rs(r+s) + s^2(s+1)).\]

But the last expression in the preceding line is positive, since by our assumption $r+s = \lambda - \mu > 0$, so both the terms $rs(r+s)$ and $s^2(s+1)$ are negative, making the expression in parentheses negative. Thus, combined with the previous inequality, we have $r(d+sf) > 0$, and since $r > 0$, also $d+sf > 0$, completing the contradiction.
 \end{proof}
 
 Greaves and Soicher (Remark 5.4 in \cite{greaves16}) note that SRGs with $\lambda = 1$ have $\omega(G) = 3$. We can prove Conjecture~\ref{con:ew} for these SRGs as follows.

 \begin{proof}
 Note first that $r = \lambda -\mu - s = 1 - \mu -s \le -s$, so $r^2 \le -rs = d - \mu < d$. We are therefore seeking to prove that:
 \begin{equation}\label{greaves}
 d^2 + 2r^2 < d^2 + 2d \le \frac{2nd}{3}.
 \end{equation}
 
 We now use the well known result that $\mu \ge 1$, and (as above) that $rs = \mu - d, r + s = \lambda - \mu$ and that $(d - r)(d - s) = n(d + rs)$.  Consequently (\ref{greaves}) simplifies to:
 \[
 d^2 + 3drs + 6d + 4rs  \le -2d(r + s) = 2d(\mu - 1) = 2d(rs + d - 1),
 \]
 which in turn simplifies to
 \[
 8d + 4rs + drs \le d^2.
 \]
 
 Since $s$ is negative this inequality is true for all $d \ge 8$. The only SRGs with $d < 8$ and $\lambda = 1$ are SRG(9, 4, 1, 2) and SRG(15, 6, 1 ,3) both of which satisfy Conjecture~\ref{con:ew}. This completes the proof.
 \end{proof}

\subsection{Applying Conjecture~\ref{con:ew}}

Conjecture~\ref{con:ew} is unusual because $\omega(G)$ appears on both sides of the inequality. It is therefore helpful to describe how the Conjecture can be used to bound $\omega(G)$, on the assumption the conjecture is true. This is best done with examples. 

Consider the graph Circulant(16,(1,2,3,4)) on 16 vertices with 64 edges with spectrum $(8^1, 4.027^2, 0^3, -0.332^2, -1.198^2, -2^4, -2.496^2).$ This graph has $\omega(G) = 5$ but for the purpose of this example we will assume we do not know this. All we can say is that $2 \le \omega(G) \le \chi \le 1 + \mu_1 = 9$. We therefore begin with $\omega(G) = 2$ and then keep increasing $\omega(G)$ until the Conjecture is satisfied. So for example since the Conjecture fails when $\omega(G) = 2$ then we know $\omega(G) > 2$. Note that $n^+ = 3$. Therefore with 

\begin{itemize}
    \item $\omega(G) = 2,  \ell = 2$ so $\mu_1^2 + \mu_2^2 = 80.2 > 2m(\omega(G) - 1)/\omega(G) = 64$ which implies $\omega(G) > 2.$
    \item $\omega(G) = 3, \ell = 3$ so $\mu_1^2 + \mu_2^2 + \mu_3^2 = 96.4 > 2m(\omega(G) - 1)/\omega(G) = 85.3$ which implies $\omega(G) > 3.$
    \item $\omega(G) = 4, \ell = 3$ so $\mu_1^2 + \mu_2^2 + \mu_3^2 = 96.4 > 2m(\omega(G) - 1)/\omega(G) = 96$ which implies $\omega(G) >  4.$
    \item $\omega(G) = 5, \ell= 3$ so $\mu_1^2 + \mu_2^2 + \mu_3^2 = 96.4 < 2m(\omega(G) - 1)/\omega(G) = 102.4$ which implies $\omega(G) \ge 5.$
\end{itemize}

Therefore for this graph Conjecture~\ref{con:ew} produces a lower bound for the clique number which equals the clique number.

As a second example consider Barbell(8) with $n = 16, m = 57$ and spectrum $(7.14^1, 6.89^1, -0.14^1, -1^{12}, -1.88^1).$ It is immediate that $\omega(G) = \chi(G) = 8$, but assume we do not know this for the purposes of this example. Note that $n^+ = 2$ so $\ell = 2$ for all $\omega(G)$, and Conjecture~\ref{con:ew} reduces to Conjecture~\ref{con:bn}. Therefore with:

\begin{itemize}
    \item $\omega(G) = 7, \ell = 2$ so $\mu_1^2 + \mu_2^2 = 98.45 > 2m(\omega(G) - 1)/\omega(G) = 97.7$ which implies $\omega(G) > 7$.
    \item $\omega(G) = 8, \ell = 2$ so $\mu_1^2 + \mu_2^2 = 98.45 < 2m(\omega(G) - 1)/\omega(G) = 99.7$ which implies $\omega(G) \ge 8$.
\end{itemize}

However $\omega(G) \le \chi \le 1 + \mu_1 = 8.14$ so using spectral information alone $\omega(G) = 8$. It is worth noting that the Hoffman \cite{hoffman70} lower bound for the chromatic number gives that:

\[
1 + \frac{\mu_1}{|\mu_n|} = 1 + \frac{7.14}{1.88} = 4.8 \le \chi(G).
\]

Therefore in this case both conjectured lower bounds for the clique number outperform a well known lower bound for the chromatic number. However the full version of the Hoffman bound \cite{hoffman70} is that $\mu_1 + \mu_{n-\chi+2} + \ldots + \mu_n \le 0$, which does imply $\chi = 8$. 



\section{Conclusions}

The question posed by Conjecture 1 is whether Wilf's bound provides another example where $\mu^2$ can be replaced by $s^+$? If so, are there further bounds in which this strengthening can be made?

The question posed by Conjecture 2 is whether the sign of eigenvalues is relevant to lower bounds for the clique number, since it is this refinement to Conjecture 3 that enables Conjecture 2 to be exact for $K_n$, whereas $K_n$ is excluded in Conjecture 3. Can Zhang's \cite{zhang} elegant proof of Conjecture 3 for regular graphs be adapted to enable a proof of Conjecture 2 for regular graphs?
 
\section*{Acknowledgement}
We thank the anonymous referee for their careful reviewing of the paper.

\end{document}